\documentclass[a4j,dvipdfmx]{article}
\usepackage{amscd}
\usepackage{amsmath}
\usepackage{amssymb}
\usepackage{amsthm}
\usepackage{ascmac}
\usepackage{graphicx}
\usepackage{multicol}
\usepackage{tikz}
\theoremstyle{definition}
\newtheorem{df}{Definition}
\newtheorem{note}{Notation}
\theoremstyle{theorem}
\newtheorem{prop}{Proposition}
\newtheorem{thm}{Theorem}
\newtheorem{lem}{Lemma}
\newtheorem{sublem}{Sublemma}

\title{Average-sized miniatures and normal-sized miniatures\\
of lattice polytopes}
\author{Takashi HIROTSU}
\date{\today}
\begin{document}
\maketitle
\begin{abstract}
Let $d \geq 0$ be an integer and let $P \subset \mathbb R^d$ be a $d$-dimensional lattice polytope. 
We call a polytope $M \subset \mathbb R^d$ such that $M \subset P$ and $M \sim P$ a {\itshape miniature} of $P,$ and it is said to be {\itshape horizontal} if $M$ is transformed into $P$ by translating and positive integral rescaling. 
A miniature $M$ of $P$ is said to be {\itshape average-sized} (resp.~{\itshape normal-sized}) if the volume of $M$ is equal to the limit of the sequence whose $n$-th term is the average of the volumes of all miniarures (resp.~all horizontal miniatures) whose vertices belong to $(n^{-1}\mathbb Z)^d.$ 
We prove that, for any lattice square $P \subset \mathbb R^2,$ the ratio of the areas of an average-sized miniature of $P$ and $P$ is $2:15.$ 
We also prove that, for any lattice simplex $P \subset \mathbb R^d,$ the ratio of the volume of a normal-sized miniature of $P$ to that of $P$ is $1:\binom{2d+1}{d}.$ 
This ratio is same as the known result for the hypercube $[0,1]^d$ provided by the author.
\end{abstract}
\section{Introduction}
Let $d \geq 0$ be an integer. 
The author proved the following theorem in \cite{Hir}.
\begin{thm}[{\cite[Theorem 2]{Hir}}]\label{thm-1}
The sequence whose $n$-th term is the average of the volumes of all hypercubes in $[0,1]^d$ whose vertices belong to $(n^{-1}\mathbb Z)^d$ and whose edges are parallel to those of $[0,1]^d$ converges to $\binom{2d+1}{d}^{-1}.$
\end{thm}
To consider a problem similar to this theorem for other shapes, we introduce the concepts below. 
In this article, we refer to a $d$-dimensional convex polytope in $\mathbb R^d$ as a polytope.
\begin{note}
For each polytope $P \subset \mathbb R^d,$ vector $a \in \mathbb R^d,$ and real number $c > 0,$ we define 
\[ P+a = \{ x+a \mid x \in P\} \quad\text{and}\quad cP = \{ cx \mid x \in P\}.\]
If polytopes $P,$ $P' \subset \mathbb R^d$ are similar, we write $P \sim P'.$
\end{note}
\begin{df}
Let $P \subset \mathbb R^d$ be a polytope.
\begin{enumerate}
\item[(1)]
We call a polytope $M \subset \mathbb R^d$ a {\itshape miniature} of $P$ if $M \subset P$ and $M \sim P.$
\item[(2)]
A miniature $M$ of $P$ is said to be {\itshape horizontal} if $M$ is transformed into $P$ by translating and positive integral rescaling.
\end{enumerate}
\end{df}
\begin{df}
Let $P \subset \mathbb R^d$ be a polytope, and let $M$ be a miniature of $P.$
\begin{enumerate}
\item[(1)]
We call $P$ (resp.~$M$) a {\itshape lattice polytope} (resp.~{\itshape lattice miniature}) if any vertex of $P$ (resp.~$M$) belongs to $\mathbb Z^d.$
\item[(2)]
We call a horizontal lattice miniature of $P$ with minimum volume a {\itshape fundamental miniature} of $P.$
\item[(3)]
A lattice polytope $P$ is said to be {\itshape irreducible} if $P$ is the only fundamental miniature of $P,$ and $P$ is said to be {\itshape reducible} otherwise.
\end{enumerate}
\end{df}
\begin{df}
Let $P \subset \mathbb R^d$ be a lattice polytope, let $M$ be a miniature of $P,$ and let $n > 0$ be an integer.
\begin{enumerate}
\item[(1)]
Then $P$ (resp.~$M$) is said to be {\itshape rational} if any vertex of $P$ (resp.~$M$) belongs to $\mathbb Q^d.$
\item[(2)]
We call $M$ a miniature of $P$ with {\itshape resolution} $n$ if any vertex of $M$ belongs to $(n^{-1}\mathbb Z)^d.$
\end{enumerate}
\end{df}
\begin{df}
For each lattice polytope $P \subset \mathbb R^d,$ let $\mathcal M_n(P)$ denote the set of all miniatures of $P$ with resolution $n.$ 
If the limit 
\[\mu _{\mathrm{av}}(P) := \lim_{n \to \infty}\frac{\sum_{M \in \mathcal M_n(P)}\mathrm{vol}(M)}{\#\mathcal M_n(P)}\]  
exists in $\mathbb R,$ then $P$ is said to be {\itshape $\mu_{\mathrm{av}}$-measurable}, and we call a miniature of $P$ with volume $\mu _{\mathrm{av}}(P)$ an {\itshape average-sized miniature} of $P.$
\end{df}
\begin{df}
For each lattice polytope $P \subset \mathbb R^d,$ let $\mathcal H_n(P)$ denote the set of all horizontal miniatures of $P$ with resolution $n.$ 
If the limit 
\[\mu _{\mathrm{nl}}(P) := \lim_{n \to \infty}\frac{\sum_{M \in \mathcal H_n(P)}\mathrm{vol}(M)}{\#\mathcal H_n(P)}\]  
exists in $\mathbb R,$ then $P$ is said to be {\itshape $\mu_{\mathrm{nl}}$-measurable}, and we call a miniature of $P$ with volume $\mu _{\mathrm{nl}}(P)$ a {\itshape normal-sized miniature} of $P.$
\end{df}
The following property is fundamental.
\begin{prop}\label{prop-1}
Let $P,$ $P' \subset \mathbb R^d$ be lattice polytopes. 
If $P \sim P'$ and $P$ is $\mu_{\mathrm{av}}$-measurable (resp.~$\mu_{\mathrm{nl}}$-measurable), then $P'$ is also $\mu_{\mathrm{av}}$-measurable (resp.~$\mu_{\mathrm{nl}}$-measurable), and we have 
\[\frac{\mu_{\mathrm{av}}(P)}{\mathrm{vol}(P)} = \frac{\mu_{\mathrm{av}}(P')}{\mathrm{vol}(P')} \quad \left( resp.~\frac{\mu_{\mathrm{nl}}(P)}{\mathrm{vol}(P)} = \frac{\mu_{\mathrm{nl}}(P')}{\mathrm{vol}(P')}\right).\]
\end{prop}
The proof is given in Section \ref{sec-prop}. 
The first main result of this article is the following theorem, which provides an example of an average-sized miniature of a polytope. 
\begin{thm}\label{thm-2}
Let $P \subset \mathbb R^2$ be a lattice square with area $A.$ 
Then $P$ is $\mu_{\mathrm{av}}$-measurable, and we have 
\begin{equation} 
\mu _{\mathrm{av}}(P) = \frac{2}{15}A. \label{eq-av-sq} 
\end{equation}
\end{thm}
The proof is given in Section \ref{sec-av}. 
The second main result of this article is the following theorem.
\begin{thm}\label{thm-3}
Let $P \subset \mathbb R^d$ be a lattice simplex. 
Then $P$ is $\mu_{\mathrm{nl}}$-measurable, and we have 
\[\mu_{\mathrm{nl}}(P) = \binom{2d+1}{d}^{-1}\mathrm{vol}(P).\]
\end{thm}
The proof is given in Section \ref{sec-nl}.
\section{Miniatures of Similar Lattice Polytopes}\label{sec-prop}
In this section, we prove Proposition \ref{prop-1}. 
We use the following lemmas.
\begin{lem}\label{lem-1}
Let $P \subset \mathbb R^d$ be a $\mu_{\mathrm{av}}$-measurable (resp.~$\mu_{\mathrm{nl}}$-measurable) lattice polytope, and let $a \in \mathbb Z^d.$ 
Then $P+a$ is also $\mu_{\mathrm{av}}$-measurable (resp.~$\mu_{\mathrm{nl}}$-measurable), and we have 
\[\mu_{\mathrm{av}}(P+a) = \mu_{\mathrm{av}}(P) \quad (\text{resp}.~\mu_{\mathrm{nl}}(P+a) = \mu_{\mathrm{nl}}(P)).\]
\end{lem}
\begin{proof}
For each integer $n > 0,$ since every point in $P\cap (n^{-1}\mathbb Z)^d$ is mapped to a point in $(P+a)\cap (n^{-1}\mathbb Z)^d$ by the transformation $x \mapsto x+a,$ we have 
\[\mathcal M_n(P+a) = \{ M+a \mid M \in \mathcal M_n(P)\} \quad\text{and}\quad \mathcal H_n(P+a) = \{ M+a \mid M \in \mathcal H_n(P)\},\]
which imply the desired assertions.
\end{proof}
\begin{lem}\label{lem-2}
Let $P \subset \mathbb R^d$ be a $\mu_{\mathrm{av}}$-measurable (resp.~$\mu_{\mathrm{nl}}$-measurable) lattice polytope, and let $c > 0$ be an integer. 
Then $cP$ is also $\mu_{\mathrm{av}}$-measurable (resp.~$\mu_{\mathrm{nl}}$-measurable), and we have 
\[\mu_{\mathrm{av}}(cP) = c^d\mu_{\mathrm{av}}(P) \quad (resp.~\mu_{\mathrm{nl}}(cP) = c^d\mu_{\mathrm{nl}}(P)).\]
\end{lem}
\begin{proof}
For each integer $n > 0,$ since every point in $P\cap (c^{-1}n^{-1}\mathbb Z)^d$ is mapped to a point in $cP\cap (n^{-1}\mathbb Z)^d$ by the transformation $x \mapsto cx,$ we have 
\[\mathcal M_n(cP) = \{ cM \mid M \in \mathcal M_{cn}(P)\} \quad\text{and}\quad \mathcal H_n(cP) = \{ cM \mid M \in \mathcal H_{cn}(P)\},\]
which imply the desired assertions.
\end{proof}
\begin{lem}\label{lem-3}
Let $P,$ $P' \subset \mathbb R^d$ be irreducible lattice polytopes. 
If $P \sim P'$ and $P$ is $\mu_{\mathrm{av}}$-measurable (resp.~$\mu_{\mathrm{nl}}$-measurable), then $P'$ is also $\mu_{\mathrm{av}}$-measurable (resp.~$\mu_{\mathrm{nl}}$-measurable), and we have 
\[\frac{\mu_{\mathrm{av}}(P)}{\mathrm{vol}(P)} = \frac{\mu_{\mathrm{av}}(P')}{\mathrm{vol}(P')} \quad \left(\text{resp}.~\frac{\mu_{\mathrm{nl}}(P)}{\mathrm{vol}(P)} = \frac{\mu_{\mathrm{nl}}(P')}{\mathrm{vol}(P')}\right).\]
\end{lem}
\begin{proof}
This follows from the following sublemma.
\end{proof}
\begin{sublem}
Let $P,$ $P' \subset \mathbb R^d$ be lattice polytopes. 
Suppose that $P$ is irreducible, $P \sim P',$ and $P$ is transformed into $P'$ by a lattice isomorphism $\tau :\mathbb R^d\to\mathbb R^d.$ 
\begin{enumerate}
\item[{\rm (1)}]
Then we have 
\[\mathcal M_n(P') \supset \{\tau (M) \mid M \in \mathcal M_n(P)\} \quad\text{and}\quad \mathcal H_n(P') \supset \{\tau (M) \mid M \in \mathcal H_n(P)\}.\] 
\item[{\rm (2)}]
Furthermore, if $P'$ is irreducible, then 
\[\mathcal M_n(P') = \{\tau (M) \mid M \in \mathcal M_n(P)\} \quad\text{and}\quad \mathcal H_n(P') = \{\tau (M) \mid M \in \mathcal H_n(P)\}.\]
\end{enumerate}
\end{sublem}
\begin{proof}
\begin{enumerate}
\item[(1)]
Let $V \subset \mathbb Z^d$ be the set of all vertices of $P.$ 
For each integer $n > 0,$ since every point $\sum_{v \in V}a_v v$ in $P\cap (n^{-1}\mathbb Z)^d,$ where $a_v \in n^{-1}\mathbb Z$ for each $v \in V,$ is mapped to $\sum_{v \in V}a_v\tau (v)$ in $P'\cap (n^{-1}\mathbb Z)^d$ by $\tau,$ we have the desired inclusions.
\item[(2)]
We prove the contraposition. 
Let $M' \in \mathcal M_n(P')\setminus\{\tau (M) \mid M \in \mathcal M_n(P)\}.$ 
Let $H'$ be a horizontal lattice miniature of $P'$ with minimum volume under the condition that $M' \subset H'.$ 
Since $P$ is irreducible, we have $\tau ^{-1}(H') \neq P$ and $H' \neq P',$ which implies that $P'$ is reducible.\qedhere
\end{enumerate}
\end{proof} 
Now we are ready to prove Proposition \ref{prop-1}.
\begin{proof}[Proof of Proposition \ref{prop-1}]
Let $P_0$ and $P'_0$ be fundamental miniatures of $P$ and $P',$ respectively. 
Then there exist vectors $a,$ $a' \in \mathbb Z^d,$ integers $c,$ $c' > 0,$ and a lattice isomorphism $\tau :\mathbb R^d\to\mathbb R^d$ such that 
\[ P = c(P_0-a)+a, \quad P' = c'(P'_0-a')+a', \quad\text{and}\quad P'_0-a' = \tau (P_0-a),\] 
where the origin is a vertex of each of $P_0-a$ and $P'_0-a'.$ 
Combining Lemmas \ref{lem-1}--\ref{lem-3}, we obtain 
\begin{align*} 
&\frac{\mu_{\mathrm{av}}(P)}{\mathrm{vol}(P)} = \frac{c^d\mu_{\mathrm{av}}(P_0)}{c^d\mathrm{vol}(P_0)} = \frac{c'^d\mu_{\mathrm{av}}(P'_0)}{c'^d\mathrm{vol}(P'_0)} = \frac{\mu_{\mathrm{av}}(P')}{\mathrm{vol}(P')} \\ 
&\left(\text{resp}.~\frac{\mu_{\mathrm{nl}}(P)}{\mathrm{vol}(P)} = \frac{c^d\mu_{\mathrm{nl}}(P_0)}{c^d\mathrm{vol}(P_0)} = \frac{c'^d\mu_{\mathrm{nl}}(P'_0)}{c'^d\mathrm{vol}(P'_0)} = \frac{\mu_{\mathrm{nl}}(P')}{\mathrm{vol}(P')}\right).\qedhere 
\end{align*} 
\end{proof}
\section{Average-sized Miniatures of Lattice Squares}\label{sec-av}
\setcounter{equation}{0}
In this section, we prove Theorem \ref{thm-2}. 
We use the following lemma.
\begin{lem}\label{lem-4}
For any integer $n > 0,$ we have 
\begin{align} 
\sum_{i = 1}^ni(n+1-i)^2 = \sum_{i = 1}^ni^2(n+1-i) &= \frac{1}{12}n(n+1)^2(n+2) \label{eq-sum-2-1} \\  
\intertext{and} 
\sum_{i = 1}^ni^2(n+1-i)^3 = \sum_{i = 1}^ni^3(n+1-i)^2 &= \frac{1}{60}n(n+1)^2(n+2)(n^2+2n+2). \label{eq-sum-3-2} 
\end{align} 
\end{lem}
\begin{proof}
First, we prove the identities 
\begin{align*} 
\sum_{i = 0}^ni(n-i)^2 = \sum_{i = 0}^ni^2(n-i) &= \frac{1}{12}(n-1)n^2(n+1) \\ 
\intertext{and} 
\sum_{i = 0}^ni^2(n-i)^3 = \sum_{i = 0}^ni^3(n-i)^2 &= \frac{1}{60}(n-1)n^2(n+1)(n^2+1) 
\end{align*}
for any integer $n \geq 0.$ 
In both identities, the first equalities are obvious, and the second equalities are proved as 
\allowdisplaybreaks[1]
\begin{align*} 
\sum_{i = 0}^ni^2(n-i) &= n\sum_{i = 0}^ni^2-\sum_{i = 0}^ni^3 \\ 
&= n\cdot\frac{1}{6}n(n+1)(2n+1)-\frac{1}{4}n^2(n+1)^2 \\ 
&= \frac{1}{12}n^2(n+1)(2(2n+1)-3(n+1)) \\ 
&= \frac{1}{12}(n-1)n^2(n+1) \\ 
\intertext{and} 
\sum_{i = 0}^ni^3(n-i)^2 &= n^2\sum_{i = 0}^ni^3-2n\sum_{i = 0}^ni^4+\sum_{i = 0}^ni^5 \\ 
&= n^2\cdot\frac{1}{4}n^2(n+1)^2-2n\cdot\frac{1}{30}n(n+1)(2n+1)(3n^2+3n-1)+\frac{1}{12}n^2(n+1)^2(2n^2+2n-1) \\ 
&= \frac{1}{60}n^2(n+1)(15n^2(n+1)-4(2n+1)(3n^2+3n-1)+5(n+1)(2n^2+2n-1)) \\ 
&= \frac{1}{60}n^2(n+1)(n^3-n^2+n-1) \\ 
&= \frac{1}{60}(n-1)n^2(n+1)(n^2+1), 
\end{align*} 
respectively. 
Replacing $n$ with $n+1,$ we obtain \eqref{eq-sum-2-1} and \eqref{eq-sum-3-2}.
\end{proof}
Now we are ready to prove Theorem \ref{thm-2}.
\begin{proof}[Proof of Theorem \ref{thm-2}]
By Proposition \ref{prop-1}, we have only to prove \eqref{eq-av-sq} if $P = [0,1]^2.$ 
For each 
\[ (i,j) \in \{ (i,j) \in \mathbb Z^2 \mid 1 \leq i \leq n\text{ and }0 \leq j \leq n-i\},\] 
the number of miniatures of $[0,1]^2$ with resolution $n$ spanned by the vectors $(i/n,j/n)$ and $(j/n,-i/n)$ is $(n+1-i-j)^2,$ and the area of such a miniature is $(\sqrt{(i/n)^2+(j/n)^2})^2 = (i^2+j^2)/n^2$ (see the figure below).
\begin{center}
\includegraphics[scale=0.8]{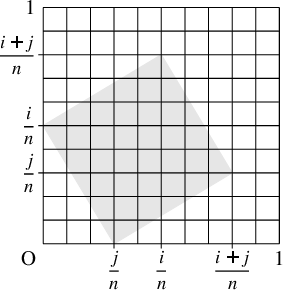}\label{fig1}
\end{center}
Hence, by \eqref{eq-sum-2-1}, we have 
\[\#\mathcal M_n([0,1]^2) = \sum_{i = 1}^n\sum_{j = 0}^{n-i}(n+1-i-j)^2 = \sum_{k = 1}^n(n+1-k)k^2 = \frac{1}{12}n(n+1)^2(n+2),\] 
where the second equality is shown by taking the diagonal sum from the top left to the bottom right in the figure below.
\begin{center}
\includegraphics[scale=0.8]{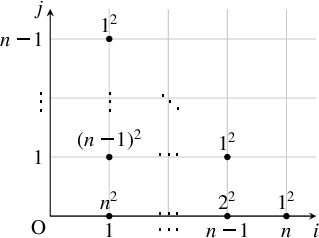}\label{fig2}
\end{center}
Furthermore, by \eqref{eq-sum-2-1} and \eqref{eq-sum-3-2}, we have 
\allowdisplaybreaks[1]
\begin{align*} 
\sum_{M \in \mathcal M_n([0,1]^2)}\mathrm{vol}(M) &= \sum_{i = 1}^n\sum_{j = 0}^{n-i}(n+1-i-j)^2\frac{i^2+j^2}{n^2} \\ 
&= \frac{1}{n^2}\sum_{k = 1}^nk^2\left(\sum_{i = 1}^{n+1-k}i^2+\sum_{j = 0}^{n-k}j^2\right) \\ 
&= \frac{1}{n^2}\sum_{k = 1}^nk^2\left( 0^2+2\sum_{i = 1}^{n-k}i^2+(n+1-k)^2\right) \\ 
&= \frac{1}{n^2}\sum_{k = 1}^nk^2\left(\frac{1}{3}(n-k)(n-k+1)(2n-2k+1)+(n+1-k)^2\right) \\ 
&= \frac{1}{3n^2}\left(\sum_{k = 1}^nk^2(n-k)(n-k+1)(2n-2k+1)+3\sum_{k = 1}^nk^2(n+1-k)^2\right) \\ 
&= \frac{1}{3n^2}\left(\sum_{k = 1}^n(n+1-k)^2(k-1)k(2k-1)+3\sum_{k = 1}^nk^2(n+1-k)^2\right) \\ 
&= \frac{1}{3n^2}\sum_{k = 1}^nk((k-1)(2k-1)+3k)(n+1-k)^2 \\ 
&= \frac{1}{3n^2}\sum_{k = 1}^nk(2k^2+1)(n+1-k)^2 \\ 
&= \frac{1}{3n^2}\left( 2\sum_{k = 1}^nk^3(n+1-k)^2+\sum_{k = 1}^nk(n+1-k)^2\right) \\ 
&= \frac{1}{3n^2}\left( 2\cdot\frac{1}{60}n(n+1)^2(n+2)(n^2+2n+2)+\frac{1}{12}n(n+1)^2(n+2)\right) \\ 
&= \frac{1}{180n^2}n(n+1)^2(n+2)(2(n^2+2n+2)+5) \\ 
&= \frac{1}{180n}(n+1)^2(n+2)(2n^2+4n+9), 
\end{align*} 
where the second equality is shown by taking the diagonal sum from the top left to the bottom right in the figure below.
\begin{center}
\includegraphics[scale=0.8]{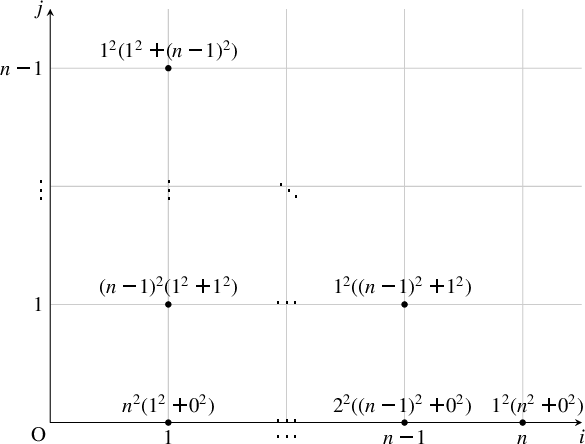}\label{fig3}
\end{center}
Thus, we obtain 
\begin{align*} 
\frac{\sum_{M \in \mathcal M_n([0,1]^2)}\mathrm{vol}(M)}{\#\mathcal M_n([0,1]^2)} &= \left.\dfrac{1}{180n}(n+1)^2(n+2)(2n^2+4n+9)\middle/\dfrac{1}{12}n(n+1)^2(n+2)\right. \\ 
&= \frac{2n^2+4n+9}{15n^2} \\ 
&\to \frac{2}{15} \quad (n \to \infty ).\qedhere 
\end{align*} 
\end{proof}
\section{Normal-sized Miniatures of Lattice Simplices}\label{sec-nl}
In this section, we prove Theorem \ref{thm-3}. 
For each integer $n > 0,$ let $S_d(n)$ denote the $n$-th {\itshape $d$-dimensional simplex number}, which is defined by 
\[ S_0(n) = 1 \quad\text{and}\quad S_d(n) = \sum_{i = 1}^nS_{d-1}(i) \quad (d > 0).\] 
The following formula is well-known.
\begin{prop}\label{prop-2}
For any integer $n > 0,$ we have 
\[ S_d(n) = \binom{n+d-1}{d}.\]
\end{prop}
We use this formula to prove the following proposition.
\begin{prop}\label{prop-3}
Let $r \geq 0$ and $n > 0$ be integers. 
Then we have 
\[\sum_{i = 1}^n i^rS_d(i) = \frac{n^{d+r+1}}{d!(d+r+1)}+O(n^{d+r}),\] 
where $O$ is the Landau symbol.
\end{prop}
\begin{proof}
This follows from 
\[ S_d(i) = \frac{1}{d!}i^d+O(i^{d-1}) \quad\text{and}\quad \sum_{i = 1}^n i^k = \frac{1}{k+1}n^{k+1}+O(n^k),\] 
where $k = d+r.$
\end{proof}
\begin{prop}[{\cite[Proposition 1]{Hir}}]\label{prop-4}
We have 
\[\binom{2d+1}{d}(d+1)\sum_{r = 0}^d\frac{(-1)^r}{d+r+1}\binom{d}{r} = 1.\]
\end{prop}
Now we are ready to prove Theorem \ref{thm-3}.
\begin{proof}[Proof of Theorem \ref{thm-3}]
Since the number of all horizontal miniatures of $P$ with similarity ratio $i/n$ is $S_d(n+1-i)$ for each integer $i$ with $1 \leq i \leq n,$ we have 
\begin{align*} 
\frac{\sum_{M \in \mathcal H_n(P)}\mathrm{vol}(M)}{\#\mathcal H_n(P)} &= \left.\left(\sum_{i = 1}^n S_d(n+1-i)\cdot\left(\frac{i}{n}\right) ^d\mathrm{vol}(P)\right)\middle/\sum_{i = 1}^n S_d(n+1-i)\right. \\ 
&= \left.\mathrm{vol}(P)\cdot\left(\sum_{i = 1}^n (n+1-i)^dS_d(i)\right)\middle/ n^d\sum_{i = 1}^n S_d(i)\right. \\ 
&= \left.\mathrm{vol}(P)\cdot\left(\sum_{i = 1}^n S_d(i)\sum_{r = 0}^d\binom{d}{r}(n+1)^{d-r}(-i)^r\right)\middle/ n^dS_{d+1}(n)\right. \\ 
&= \left.\mathrm{vol}(P)\cdot\left(\sum_{r = 0}^d (-1)^r\binom{d}{r}(n+1)^{d-r}\sum_{i = 1}^n i^rS_d(i)\right)\middle/ n^dS_{d+1}(n)\right. 
\end{align*} 
by the binomial theorem. 
Since 
\[ (n+1)^{d-r}\sum_{i = 1}^n i^rS_d(i) = \frac{1}{d!(d+r+1)}n^{2d+1}+O(n^{2d})\] 
and 
\[ n^dS_{d+1}(n) = \frac{1}{(d+1)!}n^{2d+1}+O(n^{2d}),\] 
we obtain 
\begin{align*} 
\frac{\sum_{M \in \mathcal H_n(P)}\mathrm{vol}(M)}{\#\mathcal H_n(P)} &\to \left.\mathrm{vol}(P)\cdot\left(\sum_{r = 0}^d(-1)^r\binom{d}{r}\frac{1}{d!(d+r+1)}\right)\middle/\dfrac{1}{(d+1)!}\right. \\ 
&= \mathrm{vol}(P)\cdot (d+1)\sum_{r = 0}^d\frac{(-1)^r}{d+r+1}\binom{d}{r} \\  
&= \binom{2d+1}{d}^{-1}\mathrm{vol}(P) \quad (n \to \infty ) 
\end{align*} 
by Proposition \ref{prop-4}.
\end{proof}

\end{document}